\documentclass[a4paper]{amsart}
\usepackage[T1]{fontenc}

\usepackage{amssymb}
\usepackage{amsthm}
\usepackage[ruled, lined]{algorithm2e}

\newcommand{\Aff}{\mathrm{Aff}}
\newcommand{\Inn}{\mathrm{Inn}}
\newcommand{\trid}{\triangleright}

\newcommand{\Aut}{\operatorname{Aut}}

\newcommand{\Alt}{\mathbb{A}}
\newcommand{\Sym}{\mathbb{S}}
\newcommand{\gap}{\textsf{GAP}}
\newcommand{\rig}{\textsf{Rig}}
\newcommand{\N}{\mathbb{N}}
\newcommand{\Z}{\mathbb{Z}}

\newcommand{\F}{\mathbb{F}}
\newcommand{\D}{\mathbb{D}}
\newcommand{\cT}{\mathcal{T}}

\theoremstyle{plain}
\newtheorem{thm}{Theorem}[section]
\newtheorem{lem}[thm]{Lemma}

\newtheorem{pro}[thm]{Proposition}
\newtheorem{cor}[thm]{Corollary}
\newtheorem{exa}[thm]{Example}
\newtheorem{nota}[thm]{Notation}

\newtheorem{conjeture}[thm]{Conjeture}
\theoremstyle{remark}
\newtheorem*{acknowledgement*}{Acknowledgement}

\begin{document}
\title{On the classification of quandles of low order}
\author{L. Vendramin}
\thanks{This work was partially supported by CONICET}

\address{Departamento de Matem\'atica -- FCEN,
Universidad de Buenos Aires, Pab. I -- Ciudad Universitaria (1428)
Buenos Aires -- Argentina}

\email{lvendramin@dm.uba.ar}


\date{}
\maketitle

\begin{abstract}
	Using the classification of transitive groups we classify
	indecomposable quandles of size $<36$. This classification is available
	in \rig, a \gap~package for computations related to racks and quandles.
	As an application, the list of all indecomposable quandles of size
	$<36$ not of type D is computed. 
\end{abstract}

\section{Introduction}

Racks appeared for the first time in \cite{MR1194995} and quandles appeared in
\cite{MR638121} and \cite{MR672410}. Racks and quandles are used in modern knot
theory because they provide good knot invariants, \cite{MR638121}. They are
also useful for the classification problem of pointed Hopf algebras because
they provide a powerful tool to understand Yetter-Drinfeld modules over groups,
see \cite{MR1994219}.  Of course, the classification of finite racks (or
quandles) is a very difficult problem.  Several papers about classifications of
different subcategories of racks have appear, see for example \cite{MR638121},
\cite{MR682881}, \cite{MR1848966}, \cite{MR2175299}, \cite{MR1994219},
\cite{MR2093034}, \cite{Clauwens}.

\medskip
In this paper we use the classification of transitive groups and 
the program described in  \cite{MR2414453} to classify indecomposable quandles. With this
method, we complete the classification of all non-isomorphic indecomposable
quandles of size $<36$.  This classification is available in \rig, a
\gap~\cite{gap} package designed for computations related to racks and
quandles. \rig~is a free software and it is available at
\verb+http://code.google.com/p/rig/+.

\section{Definitions and examples}

We recall basic notions and facts about racks. For additional information we
refer for example to \cite{MR1994219}.  A \textit{rack} is a pair $(X,\trid)$,
where $X$ is a non-empty set and $\trid:X\times X\to X$ is a map (considered as
a binary operation on $X$) such that
\begin{enumerate}
  \item the map $\varphi_i:X\to X$, where $x\mapsto i\trid x$, is bijective for all $i\in X$, and
  \item $i\trid(j\trid k)=(i\trid j)\trid(i\trid k)$ for all $i,j,k\in X$.
\end{enumerate}

A rack $(X,\trid )$, or shortly $X$, is a \textit{quandle} if $i\trid i=i$ for all $i\in X$.
A \textit{subrack} of a rack $X$ is a non-empty subset $Y\subseteq X$ such that
$(Y,\trid)$ is also a rack. 

\begin{exa}
A group $G$ is a quandle with $x\trid y=xyx^{-1}$ for all $x,y\in G$.  
If a subset $X\subseteq G$ is stable under conjugation by $G$, then it is a
\emph{subquandle} of $G$. 
\end{exa}

To construct racks associated to (union of) conjugacy classes of groups use the
\rig~function \verb+Rack+. 
For example, to construct the quandle of three
elements associated to the conjugacy class of transpositions in $\Sym_3$:
	\begin{verbatim}
	     gap> r := Rack(SymmetricGroup(3), (1,2));;
	     gap> Size(r);
	     3
	\end{verbatim}

\begin{exa}
	Let $G$ be a group and $s\in\Aut(G)$. Define $x\trid y=xs(x^{-1}y)$ for
	$x,y\in G$.  Then $(G,\trid)$ is a quandle.  Further, let $H\subseteq
	G$ be a subgroup such that $s(h)=h$ for all $h\in H$.  Then $G/H$ 
	is a quandle with $xH\trid yH=xs(x^{-1}y)H$. It is called the
	\emph{homogeneous quandle} $(G,H,s)$.
\end{exa}

\begin{exa}
	Let $n\geq2$. The \emph{dihedral quandle} of order $n$ is
	$\Z_n=\{0,1,\dots,n-1\}$ with $i\trid j=2i-j\pmod{n}$. 
\end{exa}

The package provides several functions to construct racks and quandles. See the
documentation for more information.

\medskip
Let $X$ be a finite rack. Assume that $X=\{x_1,x_2,\dots,x_n\}$. With the
identification $x_i\equiv i$ the rack $X$ can be presented as a square matrix
$M\in\N^{n\times n}$ such that $M_{ij}=(i\trid j)$. This matrix is called
\emph{the table} of the rack. See \cite{MR2175299}.

\begin{exa}
The matrix (or table) of the rack $\D_4$ is
\[
\begin{pmatrix}
\varphi_1\\
\varphi_2\\
\varphi_3\\
\varphi_4
\end{pmatrix}
=\begin{pmatrix}
1&4&3&2\\
3&2&1&4\\
1&4&3&2\\
3&2&1&4
\end{pmatrix}.
\]
The files of the matrix are the permutations of the quandle: $\varphi_1=\varphi_3=(2\,4)$ and $\varphi_2=\varphi_4=(1\,3)$.
\begin{verbatim}
     gap> D4 := DihedralQuandle(4);;  
     gap> Permutations(D4);
     [ (2,4), (1,3), (2,4), (1,3) ]
     gap> Table(D4);
     [ [  1,  4,  3,  2 ],
       [  3,  2,  1,  4 ],
       [  1,  4,  3,  2 ],
       [  3,  2,  1,  4 ] ]
\end{verbatim}
\end{exa}

Let $(X,\trid)$ and $(Y,\trid)$ be racks.  A map $f:X\to Y$ is a
\textit{morphism} of racks if $f(i\trid j)=f(i)\trid f(j)$ for all $i,j\in X$.

\begin{nota}
We write $g^G$ for the conjugacy class of $g$ in $G$. 
\end{nota}

\begin{exa}
Let  $\cT_1=(1\,2\,3)^{\Alt_4}$ and
$\cT_2=(1\,3\,2)^{\Alt_4}$. Then the quandles $\cT_1$ and $\cT_2$ are isomorphic.
\begin{verbatim}
     gap> T1 := Rack(AlternatingGroup(4), (1,2,3));;
     gap> T2 := Rack(AlternatingGroup(4), (1,3,2));;
     gap> IsomorphismRacks(T1, T2);
     (3,4)
\end{verbatim}
Hence $\cT_1\simeq\cT_2$ and the isomorphism is given by the permutation
$\sigma=(3\,4)$.  More precisely, assume that $\cT_1=\{x_1,x_2,x_3,x_4\}$ and
$\cT_2=\{y_1,y_2,y_3,y_4\}$.  Then the map $f:\cT_1\to \cT_2$,
$f(x_i)=y_{\sigma(i)}$, is an isomorphism of racks. 
\end{exa}
 

\begin{exa}
Let $A$ be an abelian group, and let $T\in\Aut(A)$. We have a quandle structure on
$A$ given by
\[
a\trid b=(1-T)a+Tb
\]
for $a,b\in A$. The quandle $(A,\trid)$ is called \emph{affine (or Alexander)
quandle} and it will be denoted by $\mathrm{Aff}(A,T)$.  
In particular, let $p$ be a prime number, $q$ a power of $p$ 
and $\alpha\in\F_q^\times=\F_q\setminus\{0\}$.
We write $\mathrm{Aff}(\mathbb{F}_q, \alpha)$, or simply
$\mathrm{Aff}(q,\alpha)$, for the affine quandle $\mathrm{Aff}(A,g)$, where
$A=\mathbb{F}_q$ and $g$ is the automorphism given by $x\mapsto\alpha x$ for
all $x\in\F_q$.
\end{exa}

\begin{exa}
\label{exa:T}
The \emph{tetrahedron quandle} is the quandle $\cT=(1\,2\,3)^{\Alt_4}$. It is
easy to see that this quandle is isomorphic to an affine quandle over $\F_4$.
\end{exa}


The \textit{inner group} of a rack $X$ is the group generated by the
permutations $\varphi_i$ of $X$, where $i\in X$. We write $\Inn(X)$ for the
inner group of $X$.  A rack is said to be \emph{faithful} if the map
\begin{align*}
\varphi :X\to\Inn (X),\qquad
i\mapsto\varphi_i,
\label{eq:varphi}
\end{align*}
is injective. 
We say that a rack $X$ is \textit{indecomposable} (or connected) if the inner
group $\Inn(X)$ acts transitively on $X$. Also, $X$ is \textit{decomposable} if
it is not indecomposable. Any finite rack $X$ is the disjoint union of
indecomposable subracks \cite[Prop.\,1.17]{MR1994219} called the
\textit{components of} $X$.

\begin{exa}
\label{exa:D4}
The dihedral quandle $\D_4$ is decomposable: $\D_4=\{1,3\}\sqcup\{2,4\}$.
\begin{verbatim}
     gap> D4 := DihedralQuandle(4);;
     gap> IsIndecomposable(D4);
     false
     gap> Components(D4);
     [ [ 1, 3 ], [ 2, 4 ] ]
\end{verbatim}
\end{exa}

For any rack $X$, the \textit{enveloping group} of $X$ is 
\[
G_X=F(X)/\langle iji^{-1}=i\trid j,\; i,j\in X\rangle,
\]
where $F(X)$ denotes the free group generated by $X$. This group is also called the
\textit{associated group} of $X$, see \cite{MR1194995}.
Let
\[
\overline{G_X}=G_X/\langle x^{\mathrm{ord}(\varphi_x)}\mid x\in X\rangle.
\]
If $X$ is finite then the group $\overline{G_X}$ is finite and it is called the
\emph{finite enveloping group of $X$}, see \cite{marburg}. 

\begin{exa}
	Let $X=\cT$ be the tetrahedron rack. Then $\Inn(X)\simeq\Alt_4$ and
	$\overline{G_X}\simeq\mathbf{SL}(2,3)$.
	\begin{verbatim}
	     gap> T := Rack(AlternatingGroup(4), (1,2,3));; 
	     gap> inn := InnerGroup(T);;
	     gap> StructureDescription(inn);
	     A4
	     gap> env := FiniteEnvelopingGroup(T);;
	     gap> StructureDescription(env);
	     SL(2,3)
	\end{verbatim}
\end{exa}

Table \ref{tab:finite_enveloping_groups} contains the inner group and
the finite enveloping groups associated to some particular racks.
These racks appear in the classification of finite-dimensional Nichols
algebras, see for example \cite[Table 6]{CLA}.

\begin{table}[h]
\caption{Some finite enveloping groups}
\begin{tabular}{|c|c|c|}
\hline 
Quandle & $\Inn(Q)$ & $\overline{G_X}$ \tabularnewline
\hline
$\D_3$ & $\Sym_3$ & $\Sym_3$ \tabularnewline
\hline
$\cT$ & $\Alt_4$ & $\mathbf{SL}(2,3)$ \tabularnewline
\hline
$\Aff(5,2)$, $\Aff(5,3)$ & $\Z_5\rtimes\Z_4$ & $\Z_5\rtimes\Z_4$ \tabularnewline
\hline
$(1\,2)^{\Sym_4}$ & $\Sym_4$ & $\Sym_4$\tabularnewline
\hline
$\Aff(7,3)$, $\Aff(7,5)$ & $(\Z_7\rtimes\Z_3)\rtimes\Z_2$ & $(\Z_7\rtimes\Z_3)\rtimes\Z_2$ \tabularnewline
\hline
$(1\,2\,3\,4)^{\Sym_4}$ & $\Sym_4$ & $\mathbf{SL}(2,3)\rtimes\Z_4$\tabularnewline
\hline
$(1\,2)^{\Sym_5}$ & $\Sym_5$ & $\Sym_5$\tabularnewline
\hline
\end{tabular}
\label{tab:finite_enveloping_groups} 
\end{table}

\section{The classification of indecomposable quandles of low order}

The main tool for the classification of indecomposable quandles is the
following theorem of \cite{MR2414453}. Our proof is heavily based on
\cite[Theorem 7.1]{MR638121}. For completeness we give a proof in the context
of this paper.



\begin{thm}
\label{thm:main}
Let $X$ be an indecomposable quandle of $n$ elements. Let $x_0\in X$, $z=\varphi_{x_0}$, 
$G=\Inn(X)$ and $H=\mathrm{Stab}_G(x_0)=\{g\in G\mid g\cdot x_0=x_0\}$. Then
\begin{enumerate}
\item $G$ is a transitive group of degree $n$,
\item $z$ is a central element of $H$,
\item $X$ is isomorphic to the homogeneous quandle $(G,H,I_z)$, where
	$I_z:G\to G$ is the conjugation $x\mapsto zxz^{-1}$. 
\end{enumerate} 
\end{thm}

\begin{proof}
The claim (1) follows by definition.  
The claim (2) follows from \cite[Theorem 4.3]{MR2414453}. 
We now prove (3). We consider the quandle structure over $G$ given by
$x\trid y=xI_z(x^{-1}y)$ for all $x,y\in G$,  
and let $e:G\to X$, $x\mapsto x\cdot x_0$, be the evaluation map. Since $G$
acts transitively on $X$, the map $e$ is surjective. We claim that $e$ is a
rack morphism.  Indeed, 
    \begin{align*}
	    e(x\trid y)=e(xs(x^{-1}y))&=e(xzx^{-1}yz^{-1})=zxx^{-1}yz^{-1}\cdot x_0\\
	    &=xzx^{-1}y\cdot x_0=x\cdot (x_0\trid (x^{-1}y\cdot x_0))=e(x)\trid e(y)
    \end{align*}
    for all $x,y\in G$. Further, $e(x)=e(y)$ if and only if $xH=yH$.
    Then $e$ induces the isomorphism $G/H\to X$, $xH\mapsto e(x)$. Hence the claim follows. 
\end{proof}

\begin{algorithm}[H]
\caption{Indecomposable quandles of size $n$}
\label{alg:quandles}\BlankLine 
\KwResult{The list $L$ of all non-isomorphic indecomposable quandles}
$L\longleftarrow\emptyset$\;
\For{all transitive groups $G$ of degree $n$}{
  Compute $H=\mathrm{Stab}_G(x_0)$\;
  Compute $Z(H)$, the center of $H$\;
  \For{$z\in Z(H)\setminus\{1\}$}{
    Compute the homogeneous quandle $Q=(G,H,I_z)$\;
    \If{$Q$ is indecomposable and $Q\not\simeq X$ for all $X\in L$}{
         Add the quandle $Q$ to $L$\;
    }
}
}
\end{algorithm}

Recall that all indecomposable quandles of prime order $p$ are affine, see
\cite{MR1848966}. Let $n\in\N$, $n<36$, and $n$ not being a prime number.  
Using Theorem \ref{thm:main} and  
Algorithm \ref{alg:quandles}, the list of all non-isomorphic 
indecomposable quandles can be constructed.  
The only requirement is the classification of transitive groups.
The complete list of transitive groups up to degree $<32$ is included in \gap.
Hulpke classified several of these transitive groups, see \cite{MR2168238}.
Further, Hulpke classified transitive groups of degree $33$, $34$ and $35$.
Transitive groups of degree $32$ were classified in \cite{MR2455702}. 


\medskip
For $n\in\N$ let $q(n)$ be the number of non-isomorphic indecomposable quandles
or size $n$. In Example \ref{exa:q(20)} above, $q(20)$ is computed. Further, Table
\ref{tab:all} shows the value of $q(n)$ for $n\in\{1,2,\dots,35\}$.

\begin{exa}
\label{exa:q(20)}
There are $10$ isomorphism classes of indecomposable quandles of order $20$. 
\begin{verbatim}
     gap> NrSmallQuandles(20);
     10
\end{verbatim}
\end{exa}

\rig~contains a huge database with the set of representatives of isomorphism
classes of indecomposable quandles of size $<36$. Let $n\in\{1,2,\dots,35\}$
such that $q(n)\ne0$, and let 
\[
Q_{n,1},Q_{n,2},\dots,Q_{n,q(n)}
\]
be the set of representatives of isomorphism classes of indecomposable quandles
of size $n$. In the package, a representative $Q_{n,i}$, $1\leq i\leq q(n)$,
can be obtained with the function \verb+SmallQuandle+.

\begin{exa}
There exists only one (up to isomorphism) indecomposable quandle of order $10$.
Further, this quandle is isomorphic to the conjugacy class of transpositions in
$\Sym_5$.
\begin{verbatim}
     gap> NrSmallQuandles(10);
     1
     gap> Q := SmallQuandle(10, 1);;
     gap> R := Rack(SymmetricGroup(5), (1,2));;
     gap> IsomorphismRacks(Q, R);
     (3,5,6,10,8,4,9,7)
\end{verbatim}
\end{exa}

Recall that a \textit{crossed set} is a quandle $(X,\trid)$ which further
satisfies $j\trid i = i$ whenever $i\trid j = j$ for all $i,j\in X$.


\begin{exa}
	It is easy to see that the only indecomposable quandles of size $<36$
	which are not crossed sets are $Q_{30,4}$ and $Q_{30,5}$.
\end{exa}

\begin{table}[h]
\caption{The number of non-isomorphic indecomposable quandles}
\begin{tabular}{c|cccccccccccc}
$n$ & 1 & 2 & 3 & 4 & 5 & 6 & 7 & 8 & 9 & 10 & 11 & 12\tabularnewline
$q(n)$ & 1 & 0 & 1 & 1 & 3 & 2 & 5 & 3 & 8 & 1 & 9 & 10\tabularnewline
\hline 
$n$ & 13 & 14 & 15 & 16 & 17 & 18 & 19 & 20 & 21 & 22 & 23 & 24\tabularnewline
$q(n)$ & 11 & 0 & 7 & 9 & 15 & 12 & 17 & 10 & 9 & 0 & 21 & 42\tabularnewline
\hline 
$n$ & 25 & 26 & 27 & 28 & 29 & 30 & 31 & 32 & 33 & 34 & 35 & \tabularnewline
$q(n)$ & 34 & 0 & 65 & 13 & 27 & 24 & 29 & 17 & 11 & 0 & 15 & \tabularnewline
\end{tabular}
\label{tab:all}
\end{table}

\begin{conjeture}
Let $p$ be an odd prime number and let $Q$ be an indecomposable quandle of $2p$
elements. Then $p\in\{3,5\}$.
\end{conjeture}

\section{Rack homology}

Let $X$ be a rack.  For $n\geq0$ let $C_n(X,\Z)=\Z X^n$. Consider
$C_*(X,\Z)$ as a complex with boundary $\partial_0=\partial_1=0$ and
$\partial_{n+1}:C_{n+1}(X,\Z)\to C_n(X,\Z)$ defined by 
\begin{align*}
&&\partial_{n+1}(x_1,x_2,\dots,x_{n+1})=\sum_{i=1}^n(-1)^{i+1}&[(x_1,\dots,x_{i-1},x_{i+1},\dots,x_{n+1})\\
&&&-(x_1,\dots,x_{i-1},x_i\trid x_{i+1},\dots,x_i\trid x_{n+1})]
\end{align*}
for $n\geq1$. It is straightforward to prove that $\partial^2=0$. The
\emph{homology} $H_*(X,\Z)$ of $X$ is the homology of the complex $C_*(X,\Z)$.
See for example  \cite{MR1990571}, \cite{MR2063665}, \cite{MR2255194} for
applications to the theory of knots and \cite{MR1994219} for applications to
the theory of Hopf algebras.

\begin{exa}
\label{exa:D3_H2}
Let $X=\D_5$. Then $H_2(X,\Z)\simeq\Z$. 
\begin{verbatim}
     gap> RackHomology(DihedralQuandle(5), 2);
     [ 1, [  ] ]
\end{verbatim}
\end{exa}

\begin{exa}
\label{exa:type23_S5_H2}
Let $X=(1\,2)(3\,4\,5)^{\Sym_5}$. 
Then $H_2(X,\Z)\simeq\Z\times\Z_6$. 
\begin{verbatim}
     gap> r := Rack(SymmetricGroup(5), (1,2)(3,4,5));;
     gap> RackHomology(r, 2);
     [ 1, [ 6 ] ]
\end{verbatim}
\end{exa}

\begin{exa}
\label{exa:T_H2}
Recall that $\cT$ is the tetrahedron quandle defined in Example \ref{exa:T}.
Then 
$H_2(\cT,\Z)\simeq\Z\times\Z_2$ and 
$H_3(\cT,\Z)\simeq\Z\times\Z_2\times\Z_2\times\Z_4$. 
Further, the torsion subgroup of $H_2(\cT,\Z)$ 
is generated by 
\[
\chi=\chi_{(1,2)}+\chi_{(1,3)}+\chi_{(2,1)}+\chi_{(2,3)}+\chi_{(3,1)}+\chi_{(3,2)},
\]
where 
\[
\chi_{(i,j)}(a,b)=
\begin{cases}
	1 \text{ if $(i,j)=(a,b)$},\\
	0 \text{ otherwise}.
\end{cases}
\]
Indeed,
\begin{verbatim}
     gap> T := Rack(AlternatingGroup(3), (1,2,3));; 
     gap> RackHomology(T, 2);
     [ 1, [ 2 ] ]
     gap> RackHomology(T, 3);
     [ 1, [ 2, 2, 4 ] ]
     gap> TorsionGenerators(T, 2);
     [ [ 0, 1, 1, 0, 1, 0, 1, 0, 1, 1, 0, 0, 0, 0, 0, 0 ] ]
\end{verbatim}
\end{exa}


Table \ref{tab:H2} contains the second rack homology group of all the
indecomposable quandles of size $\leq21$. Quandles with a prime number of
elements were not included in Table \ref{tab:H2} because of the following lemma
of \cite{MR2093034}.

\begin{lem}
Let $p$ be a prime number. Let $X$ be an indecomposable quandle of $p$
elements. Then $H_2(X,\Z)\simeq\Z$. 
\end{lem}

\begin{proof}
It follows from \cite[Lemma 5.1]{MR2093034} and \cite[Theorem 2.2]{MR1952425}.
\end{proof}


\begin{table}
\caption{Some homology groups}
\begin{tabular}{|c|c|}
\hline 
Indecomposable quandle $Q$ &  $H_{2}(Q,\Z)$\tabularnewline
\hline 
$Q_{4,1}$   & $\Z\times\Z_{2}$\tabularnewline
\hline 
$Q_{6,1}$   & $\Z\times\Z_{2}$\tabularnewline
\hline 
$Q_{6,2}$   & $\Z\times\Z_{4}$\tabularnewline
\hline 
$Q_{8,1}$, $Q_{8,2},Q_{8,3}$  & $\Z$\tabularnewline
\hline 
$Q_{9,1}$, $Q_{9,4}$, $Q_{9,5}$, $Q_{9,7}$, $Q_{9,8}$ & $\Z$\tabularnewline
\hline 
$Q_{9,2}$, $Q_{9,3}$, $Q_{9,6}$  & $\Z\times\Z_{3}$\tabularnewline
\hline 
$Q_{10,1}$   & $\Z\times\Z_{2}$\tabularnewline
\hline 
$Q_{12,1}$, $Q_{12,2}$, $Q_{12,4}$ &  $\Z\times\Z_2$\tabularnewline
\hline 
$Q_{12,3}$ &   $\Z\times\Z_{10}$\tabularnewline
\hline 
$Q_{12,5}$, $Q_{12,6}$ &  $\Z\times\Z_{4}$\tabularnewline
\hline 
$Q_{12,7}$   & $\Z\times\Z_{2}\times\Z_{4}$\tabularnewline
\hline 
$Q_{12,8}$   & $\Z\times\Z_{2}^{3}$\tabularnewline
\hline 
$Q_{12,9}$   & $\Z\times\Z_{4}^{2}$\tabularnewline
\hline 
$Q_{12,10}$   & $\Z\times\Z_{6}$\tabularnewline
\hline 
$Q_{15,1}$, $Q_{15,3}$, $Q_{15,4}$ &  $\Z$\tabularnewline
\hline 
$Q_{15,2}$   & $\Z\times\Z_{2}^{2}$\tabularnewline
\hline 
$Q_{15,5}$, $Q_{15,6}$  & $\Z\times\Z_{5}$\tabularnewline
\hline 
$Q_{15,7}$ &   $\Z\times\Z_{2}$\tabularnewline
\hline 
$Q_{16,1}$,$Q_{16,7}$   & $\Z\times\Z_{4}$\tabularnewline
\hline 
$Q_{16,2}$ &   $\Z\times\Z_{2}^{4}$\tabularnewline
\hline 
$Q_{16,3}$, $Q_{16,4}$ &   $\Z\times\Z_{2}^{2}$\tabularnewline
\hline 
$Q_{16,5}$, $Q_{16,6}$ &   $\Z\times\Z_{2}$\tabularnewline
\hline 
$Q_{16,8}$, $Q_{16,9}$ &   $\Z$\tabularnewline
\hline
$Q_{18,1}$, $Q_{18,8}$, $Q_{18,11}$, $Q_{18,12}$  & $\Z\times\Z_6$\tabularnewline
\hline
$Q_{18,2}$, $Q_{18,9}$, $Q_{18,10}$ &  $\Z\times\Z_2$\tabularnewline
\hline
$Q_{18,3}$, $Q_{18,6}$, $Q_{18,7}$ &  $\Z\times\Z_4$\tabularnewline
\hline
$Q_{18,4}$, $Q_{18,5}$ &  $\Z\times\Z_{12}$\tabularnewline
\hline
$Q_{20,1}$, $Q_{20,2}$, $Q_{20,3}$ &  $\Z\times\Z_{6}$\tabularnewline
\hline
$Q_{20,4}$, $Q_{20,7}$, $Q_{20,8}$ & $\Z\times\Z_{2}$\tabularnewline
\hline
$Q_{20,5}$, $Q_{20,9}$ &  $\Z\times\Z_{2}^2$\tabularnewline
\hline
$Q_{20,6}$ &  $\Z\times\Z_{2}\times\Z_4$\tabularnewline
\hline
$Q_{20,10}$ & $\Z\times\Z_4$\tabularnewline
\hline
$Q_{21,1}$, $Q_{21,2}$, $Q_{21,3}$, $Q_{21,4}$, $Q_{21,5}$ & $\Z$\tabularnewline
\hline
$Q_{21,6}$ & $\Z\times\Z_2^2$\tabularnewline
\hline
$Q_{21,7}$, $Q_{21,8}$ & $\Z\times\Z_7$\tabularnewline
\hline
$Q_{21,9}$ & $\Z\times\Z_2$\tabularnewline
\hline
\end{tabular}
\label{tab:H2}
\end{table}

\section{Racks of type D}

\medskip
Recall from \cite{AMPA} that a finite rack $X$ is of type D if there exists an
indecomposable subrack $Y=R\sqcup S$ (here $R$ and $S$ are the components of
$Y$) such that
\[
r\trid(s\trid(r\trid s))\ne s
\]
for some $r\in R$ and $s\in S$. 

\medskip
Quandles of type D are very important for the classification of finite-dimensional
pointed Hopf algebras, see for example the program described in \cite[\S
2.6]{CLA}. For some interesting applications we refer to \cite{AMPA,
MR2745542}.

\begin{pro}
	\label{pro:typeD}
	Let $Q$ be an indecomposable quandle of size $<36$. Then $Q$ is of type
	D if and only if $Q$ is isomorphic to one of the following quandles:
	\begin{enumerate}
		\item $Q_{12,1}$,
		\item $Q_{18,i}$ for $i\in\{1,2,3,4,5,6,7,8,9,10\}$,
		\item $Q_{20,3}$,
		\item $Q_{24,i}$ for $i\in\{1,2,3,4,5,6,8,10,11,16,17,21,22,23,26,27,28,32\}$,
		\item $Q_{27,i}$ for $i\in\{1,14\}$,
		\item $Q_{30,i}$ for $i\in\{1,2,3,4,5,6,11,12,13,14,15,16\}$,
		\item $Q_{32,i}$ for $i\in\{1,2,3,5,6,7,8,9\}$.
	\end{enumerate}
\end{pro}

\begin{proof}
By \cite{MR1848966}, indecomposable quandles of size $p$ are affine.
Further, \cite[Prop. 4.2]{CLA} implies that  affine quandles with $p$ elements are
not of type D.  Therefore we may assume that the size of $Q$ is not a prime
number. Now the claim follows from a straightforward computer calculation.
\end{proof}

\begin{cor}
	Let $Q$ be an indecomposable simple quandle of size $<36$. Assume that 
	$Q$ is of type D. Then $Q\simeq Q_{30,3}$.\qed
\end{cor}

\begin{acknowledgement*}
I am grateful to M. Gra\~na for writing several functions for the package. I
would like to thank N. Andruskiewitsch, E. Clark, F. Fantino, M. Farinati, J.
A.  Guccione, J. J Guccione, I. Heckenberger and A. Lochmann for several
conversations related to racks and quandles.  I also thank J. A. Hulpke for the
list of  transitive groups of degree 33, 34 and 35 and D. Holt for the list of
transitive groups of degree 32. 
\end{acknowledgement*}


\def\cprime{$'$}

\end{document}